\documentclass[a4paper]{amsproc}
\usepackage{amsmath,amssymb,anysize,float,amscd,graphicx}

\usepackage{caption}
\newtheorem{theorem}{Theorem}[section] 
\newtheorem{lemma}[theorem]{Lemma}     
\newtheorem{corollary}[theorem]{Corollary}
\newtheorem{construction}[theorem]{Construction}
\newtheorem{definition}[theorem]{Definition}
\newtheorem{remark}[theorem]{Remark}

\newcommand{\PG}{\mathsf{PG}}
\newcommand{\GF}{\mathsf{GF}}
\newcommand{\Knarr}{\mathcal{K}}
\newcommand{\half}{\tfrac{1}{2}}
\newcommand{\scalarproduct}{\bullet}
\renewcommand{\le}{\leqslant}

\title{Every flock generalised quadrangle has a hemisystem}
 
\author{John Bamberg, Michael Giudici and Gordon F. Royle}


\begin{document}
\maketitle

\begin{abstract}
We prove that every flock generalised quadrangle contains a hemisystem, and we provide a
construction method which unifies our results with the examples of Cossidente and Penttila
in the classical case.
\end{abstract}

%
%

\section{Introduction}\label{section:intro}

In 1965, Segre \cite{Segre65} introduced the notion of a \textit{hemisystem}, that is a
set of lines which contains exactly half the lines on each point, in his work on
\textit{regular systems} of the Hermitian surface.  Segre proved that there is just one
hemisystem of lines (up to equivalence) of the classical generalised quadrangle
$\mathsf{H}(3,3^2)$, and thirty years later, it was conjectured by J. A. Thas
\cite[pp. 333]{Thas95} that there are no hemisystems of $\mathsf{H}(3,q^2)$ for $q>3$.
However, forty years after Segre's seminal paper, Cossidente and Penttila
\cite{CossidentePenttila05} constructed, for each odd prime power $q$, a hemisystem of
$\mathsf{H}(3,q^2)$. Since this work, hemisystems have received renewed attention in
finite geometry (see \cite[\S 4]{Thas07}).  The concept of a hemisystem was extended by
Cameron, Delsarte and Goethals \cite{CameronDelsarteGoethals79} to all generalised
quadrangles with the same parameters $(s^2,s)$ as the Hermitian surface, and in addition,
they showed that such an object gives rise to a \textit{partial quadrangle}, and hence
also, to a strongly regular graph.

The only known generalised quadrangles of order $(s^2,s)$, $s$ odd, are the \textit{flock
  generalised quadrangles}, and we know that for $s$ prime, these are the only elation generalised
quadrangles with these parameters (see \cite{BambergPenttilaSchneider}). There are nine
known infinite families of such generalised quadrangles, and at least forty other examples
for $q\le 61$ (see \cite{Betten}).  As far as the authors are aware, the known hemisystems
of these generalised quadrangles are listed in \cite{BambergDeClerckDurante09},
\cite{BKLP07}, \cite{CossidentePenttila09} and \cite{CossidentePenttila05}.  In
particular, only one hemisystem of a nonclassical flock generalised quadrangle has
previously been constructed (in \cite{BambergDeClerckDurante09}). It is the main purpose
of this work to establish the following:

\begin{theorem}\label{thm:intro}
Every flock generalised quadrangle of order $(s^2,s)$, $s$ odd, contains a hemisystem.
\end{theorem}

If $\mathcal{H}$ is a hemisystem of a generalised quadrangle of order $(s^2,s)$, then we
can build a partial quadrangle by taking the point-set to be $\mathcal{H}$, and the
line-set to be the points of the generalised quadrangle.  The point-graph of a partial
quadrangle is strongly regular and the small number of constructions of such graphs are
outlined later in this section.  The complement of a hemisystem is again a hemisystem,
however, the two strongly regular graphs arising may not necessarily be isomorphic (see
\cite{BambergDeClerckDurante09} for an example of such an instance).  A hemisystem also
gives rise to a cometric $4$-class association scheme that is not metric, which are indeed
rare in the literature (see \cite{MuzychukvanDam}).

We will explain precisely what the flock generalised quadrangles are in Section
\ref{section:background}, but we stress here that we rely heavily on the geometric model
of Knarr \cite{Knarr92}. This yields flock generalised quadrangles from substructures
known as \textit{BLT-sets} of the symplectic polar space
$\mathsf{W}(5,q)$\footnote{Strictly speaking, a BLT-set is a set of lines of
  $\mathsf{W}(3,q)$, but we can identify them with a set of planes incident with a point
  of $\mathsf{W}(5,q)$.}.  The linear BLT-set yields the classical flock generalised
quadrangle $\mathsf{H}(3,q^2)$ via this construction, and varying the BLT-set gives all
flock generalised quadrangles.  We interpret the Cossidente-Penttila hemisystems of
$\mathsf{H}(3,q^2)$ in this model, and show that analogously, varying the BLT-set produces
hemisystems of the resulting nonclassical flock generalised quadrangles.  Our construction
produces a considerable number of non-isomorphic hemisystems of flock generalised
quadrangles, but computer work (reported in \cite{BambergGiudiciRoyleprep}) shows that not
every hemisystem arises from this method.

As mentioned earlier, hemisystems produce partial quadrangles.  Partial quadrangles were
introduced by Peter J. Cameron \cite{Cameron75} as a geometry of points and lines such
that every two points are on at most one line, there are no triangles, every line has the
same number $s+1$ of points, every point is incident with the same number $t+1$ of lines,
and there is a constant $\mu$ such that for every pair of noncollinear points there are
precisely $\mu$ points collinear to both. It follows directly from these conditions that
the point graph of this geometry is strongly regular with parameters\footnote{We use the
  standard notation for strongly regular graphs. The parameters $(v,k,\lambda,\mu)$ give
  us the number of vertices $v$, the valency $k$, and the constants $\lambda$ and $\mu$
  are the number of vertices adjacent to a pair of adjacent or nonadjacent vertices,
  respectively.}
$$(1+s(t+1)(\mu+st)/\mu, s(t+1), s-1 , \mu).$$ In fact, generalised quadrangles themselves
are partial quadrangles with $\mu=t+1$, that is, they satisfy the extra condition that for
every point $P$ and line $\ell$ which are not incident, there is a unique point of $\ell$
collinear with $P$. The only known partial quadrangles which are not generalised
quadrangles are:
\begin{itemize}
\item one of the seven known triangle-free strongly regular graphs,
\item one of three exceptional examples arising from caps of projective spaces,
\item those which arise from removing a point and its incident lines from a
generalised quadrangle of order $(s,s^2)$, or lastly,
\item those which arise from a \textit{hemisystem} of a
generalised quadrangle of order $(s^2,s)$, $s$ odd.
\end{itemize}
A partial quadrangle arising from a hemisystem has parameters
$\mathsf{PQ}((s-1)/2, s^2, (s-1)^2/2)$.

Finally, in the Appendix, we provide a new proof that a hemisystem of a generalised
quadrangle of order $(s^2,s)$, $s$ odd, gives rise to a partial quadrangle. The method of
\cite{CameronDelsarteGoethals79} was to show that the adjacency matrix of the point graph
had three distinct eigenvalues, whereas we adopt a more direct approach and make use of
some basic algebraic combinatorics.

%
%

\section{BLT-sets and the Knarr model}\label{section:background}

\subsection{Polar spaces and generalised quadrangles}\label{polarspaces}

By a theorem of Jacques Tits, every finite polar space is either a generalised quadrangle
or arises from a finite vector space equipped with a non-degenerate reflexive sesquilinear
form or a quadratic form. By Witt's Theorem, the maximal totally isotropic subspaces have
the same algebraic dimension, which is called the \emph{rank} of the polar space.  We will
assume that the reader is familiar with the fundamental theory of polar spaces, and we
refer the reader to \cite{HirschfeldThas} for more details.  If we take the totally
isotropic subspaces containing a given subspace of a finite polar space, we obtain the
\textit{quotient} polar space of the same type but lesser rank.  We use projective
notation for polar spaces so that they differ from the standard notation for their
collineation groups. For example, the notation $\mathsf{W}(d-1,q)$ denotes the symplectic
polar space derived from the vector space $\GF(q)^d$ equipped with a non-degenerate
alternating form. Here is a summary of the notation we will use for polar spaces.

\begin{table}[H]
\begin{tabular}{ll | ll}
\hline
Polar Space&Notation & Polar Space&Notation\\
\hline
Symplectic&$\mathsf{W}(d-1,q)$, $d$ even& Orthogonal, elliptic&$\mathsf{Q}^-(d-1,q)$, $d$ even\\
Hermitian&$\mathsf{H}(d-1,q^2)$& Orthogonal, parabolic&$\mathsf{Q}(d-1,q)$, $d$ odd\\
&&Orthogonal, hyperbolic&$\mathsf{Q}^+(d-1,q)$, $d$ even\\
\hline
\end{tabular}
\medskip
\caption{Notation for the finite polar spaces.}
\end{table}

The set of zeros of a homogeneous quadratic equation $\sum a_{ij} X_iX_j=0$ in $n$
variables over $\GF(q)$ defines a \textit{quadric} of the projective space $\PG(n-1,q)$,
and we say that a subspace is \textit{singular} with respect to the quadric if all of its
points are in the quadric. The associated bilinear form for the quadratic form $Q$
defining the quadric is given by $B_Q(u,v)=Q(u+v)-Q(u)-Q(v)$. The \textit{radical} of
$B_Q$ is the set of vectors which are orthogonal to every other vector of the ambient
vector space, and if $0$ is the only zero of $Q$ in the radical of $B_Q$, then we say that
the quadric is \textit{non-singular} and \textit{singular} otherwise.  Non-singular
quadrics come in three different types depending on the dimension and Witt index of the
quadric. In $\PG(2m,q)$, there is only one non-singular quadric up to equivalence (of the
collineation group) and it is often denoted $\mathsf{Q}(2m,q)$; the \textit{parabolic
  quadric} of $\PG(2m,q)$ with Witt index $m$.  In odd dimension, we have the
\textit{elliptic} and \textit{hyperbolic} quadrics $\mathsf{Q}^-(2m+1,q)$ and
$\mathsf{Q}^+(2m+1,q)$ respectively, and they have Witt indices $m$ and $m+1$ (resp.).

The (thick) \emph{classical generalised quadrangles} are those which arise from equipping
a vector space with an alternating, quadratic or Hermitian form, and they are
$\mathsf{W}(3,q)$, $\mathsf{Q}(4,q)$, $\mathsf{H}(3,q^2)$, $\mathsf{Q}^-(5,q)$ and
$\mathsf{H}(4,q^2)$.  By taking the totally singular points and lines of the ambient
projective space, we obtain the points and lines for the given classical generalised
quadrangle.  For example, $\mathsf{W}(3,q)$ is the incidence structure of all points of
$\PG(3,q)$ and totally isotropic lines with respect to a null polarity\footnote{A
  \textit{polarity} $\rho$ of a projective space is an involutory inclusion-reversing
  permutation. Furthermore, $\rho$ is a \textit{null polarity} if every point $P$ of the
  projective space is incident with its image $P^\rho$.}, and is a generalised quadrangle
of order $(q,q)$. The point-line dual of $\mathsf{W}(3,q)$ is $\mathsf{Q}(4,q)$, the
parabolic quadric of $\PG(4,q)$, and is therefore a generalised quadrangle of order
$(q,q)$ (see \cite[3.2.1]{FGQ}).

\subsection{Flocks of quadratic cones and BLT-sets}\label{flocks}

A \textit{quadratic cone} of $\PG(3,q)$, $q$ odd, consists of the $q+1$ lines joining a
point $v$ (the \textit{vertex}) to the points of a conic of a hyperplane not on $v$,
together with the points lying on these lines.  A \textit{flock} of the quadratic cone
$\mathcal{C}$ with vertex $v$ in $\PG(3,q)$ is a partition of the points of
$\mathcal{C}\backslash\{v\}$ into conics. J. A. Thas \cite{Thas87} showed that a flock
gives rise to a generalised quadrangle of order $(q^2,q)$, which we call a \textit{flock
  generalised quadrangle}.  A \textit{BLT-set of lines} of $\mathsf{W}(3,q)$ is a set
$\mathcal{B}$ of $q+1$ lines of $\mathsf{W}(3,q)$ such that no line of $\mathsf{W}(3,q)$
is concurrent with more than two lines of $\mathcal{B}$.  Note that a line of
$\mathsf{W}(3,q)$ not in $\mathcal{B}$ must be concurrent with either $0$ or $2$ elements
of $\mathcal{B}$, so BLT-sets are sometimes known as \textit{$(0,2)$-sets}. In \cite{BLT},
it was shown that, for $q$ odd, a flock of a quadratic cone in $\PG(3,q)$ gives rise to a
BLT-set of lines of $\mathsf{W}(3,q)$. Conversely, a BLT-set gives rise to possibly many
flocks, however we only obtain one flock generalised quadrangle up to isomorphism (see
\cite{PayneRogers90}).

For $q$ odd, Knarr \cite{Knarr92} gave a direct geometric construction of a flock
generalised quadrangle from a BLT-set of lines of $\mathsf{W}(3,q)$.  If we were to
construct a flock generalised quadrangle (see below \S\ref{Knarr}) from a \textit{linear}
BLT-set of lines of $\mathsf{W}(3,q)$ (i.e., the $q+1$ lines obtained from field reduction
of a Baer subline), we would obtain a generalised quadrangle isomorphic to the classical
object $\mathsf{H}(3,q^2)$.  The BLT-sets of lines of $\mathsf{W}(3,q)$ have been
classified by Law and Penttila \cite{LawPenttila03} for prime powers $q$ at most $29$, and
this has recently been extended by Betten \cite{Betten} to $q\le 61$.

\subsection{The Knarr model}\label{Knarr}

We will require familiarity with the symplectic polar space $\mathsf{W}(5,q)$ of rank
$3$. This geometry simply arises from taking the one, two and three-dimensional vector
subspaces of $\GF(q)^6$ for which a given alternating bilinear form restricts to the zero
form (i.e., the \textit{totally isotropic} subspaces). For example, one can take this
alternating bilinear form to be defined by
$$\beta(x , y ) = x_1y_6-x_6y_1+x_2y_5-x_5y_2+x_3y_4-x_4y_3.$$ This bilinear form also
gives us a null polarity $\perp$ of the ambient projective space $\PG(5,q)$, defined by
$U\mapsto U^\perp := \{v\in \GF(q)^6\mid\beta( u,v ) =0\text{ for all } u\in U\}$.  We
will now spend some time revising the Knarr construction of a flock generalised quadrangle
as it is the model that we use in this paper to analyse hemisystems of flock generalised
quadrangles.  The ingredients of the Knarr construction are as follows:
\begin{itemize}
\item a null polarity $\perp$ of $\mathsf{PG}(5,q)$;
\item a point $P$ of $\mathsf{PG}(5,q)$;
\item a BLT-set of lines $\mathcal{O}$ of $\mathsf{W}(3, q)$.
\end{itemize}

Note that the totally isotropic lines and planes incident with $P$ yield the quotient
polar space $P^\perp/P$ isomorphic to $\mathsf{W}(3, q)$.  So we will abuse notation and
identify $\mathcal{O}$ with a set of totally isotropic planes on $P$.  Then we construct a
generalised quadrangle $\Knarr(\mathcal{O})$ from the objects in the table below and
incidence inherited from that of $\mathsf{PG}(5,q)$.

\begin{table}[H]
\begin{tabular}{lp{5.8cm}|lp{5.8cm}}
&Points && Lines \\ \hline (i) &points of $\mathsf{PG}(5,q)$ not in $P^\perp$&(a)& totally
  isotropic planes not on $P$ and meeting some element of $\mathcal{O}$ in
  a line \\ (ii) &lines not incident with $P$ but contained in some element of
  $\mathcal{O}$&(b)& elements of $\mathcal{O}$\\ (iii)& the point $P$&&\\ \hline \\
\end{tabular}
\end{table}

The point $P$ is the \textit{base point} of the flock generalised quadrangle. It turns out
(see Payne's article \cite[Proposition 3]{Payne96}, and in particular, \cite[Chapter
  5]{MaskaThesis}), that if the flock generalised quadrangle $\Knarr(\mathcal{O})$ is
nonclassical, then its automorphism group is of the form $(q-1)\cdot(E_q\rtimes K)$ where
$E_q$ is the Heisenberg group of order $q^5$ with a centre of order $q$, $K$ is the
stabiliser of the BLT-set $\mathcal{O}$ within $\mathsf{P\Gamma Sp}(4,q)$ and the cyclic
group of order $q-1$ on the bottom gives us the kernel of the action on the lines through
the base point; actually there is an exception in the case of the \textit{Kantor-Knuth}
flock generalised quadrangles in which the kernel could be two or four times larger.

\subsection{An equivalence relation on a BLT-set}

In a generalised quadrangle, the \textit{trace} of two noncollinear points $x$ and $y$ is
the set of all points collinear to both $x$ and $y$ (see \cite[pp. 2]{FGQ}), and it is
often denoted by $\{x,y\}^\perp$ or $\mathrm{Tr}(x,y)$.  Similarly, we can define the
trace $\mathrm{Tr}(\ell,m)$ of two disjoint lines $\ell$ and $m$ as the set of all lines
concurrent with both $\ell$ and $m$. So in $\mathsf{W}(3,q)$, the trace of any pair of
disjoint lines consists of $q+1$ lines. The following lemma follows from \S1.3.6 and
\S3.3.1 of \cite{FGQ} and will be fundamental in proving Lemma \ref{lemma:eqrelation}.

\begin{lemma}\label{grid0or2}
Let $\ell_1, \ell_2, \ell_3$ be three pairwise disjoint lines of $\mathsf{W}(3,q)$, $q$
odd. Then
 $$|\mathrm{Tr}(\ell_1,\ell_2)\cap\mathrm{Tr}(\ell_1,\ell_3)|\in\{0,2\}.$$
\end{lemma}

Given a line which is not concurrent with any member of a BLT-set of lines $\mathcal{B}$
of $\mathsf{W}(3,q)$, we obtain an interesting equivalence relation on the elements of
$\mathcal{B}$ (Definition \ref{defn:eqrelation}), and in order to prove that we indeed
have such an equivalence relation, we resort to working in the dual situation in the
parabolic quadric $\mathsf{Q}(4,q)$.  The notional advantages of working in this setting
are the perceived ``extra'' geometric structures that we have at our disposal. This
duality arises from the Klein correspondence between the lines of $\mathsf{PG}(3,q)$ and
$\mathsf{Q^+}(5,q)$ (see \cite{Hirschfeld85}). For example, the dual of a regulus of
$\mathsf{W}(3,q)$ is a conic of $\mathsf{Q}(4,q)$, and in some sense, non-degenerate
subspaces are easier to work with than sets of lines. There are two isometry types for
non-degenerate lines with respect to a non-singular quadric $\mathcal{Q}$, namely, (i)
lines containing no points of $\mathcal{Q}$ (\textit{external} lines) and (ii) lines
containing two points of $\mathcal{Q}$ (\textit{secant} lines). The following lemma is
needed for the proof of Lemma \ref{lemma:eqrelation}.

\begin{lemma}\label{inthequadric}
Let $\mathcal{Q}$ be a non-singular quadric of $\PG(4,q)$, $q$ odd, and let $e$ be a line
of $\PG(4,q)$ external to $\mathcal{Q}$.  Take three distinct points $b$, $b'$ and $b''$
of the non-singular conic $e^\perp\cap\mathcal{Q}$ and define the following non-degenerate
planes:
$$\pi_1=\langle b,b'\rangle^\perp,\quad \pi_2=\langle b',b''\rangle^\perp,\quad
\pi_3=\langle b'',b\rangle^\perp.$$ Let $C$ be a degenerate hyperplane meeting $\pi_1$ in
a secant line. Then the non-degenerate lines $C\cap \pi_2$ and $C\cap \pi_3$ are of
different isometry types.
\end{lemma}

\begin{proof}
Suppose $\mathcal{Q}$ is defined by the quadratic form
$$Q(x)=x_1x_5+x_2x_4+x_3^2.$$ Without loss of generality, we may chose $e$ to be the line
$\langle (1,0,0,0,n), (0,0,1,0,0)\rangle$ where $n$ is some fixed non-square of $\GF(q)$,
and since the stabiliser of a conic in $\mathsf{PO}(5,q)$ is $3$-transitive on the conic,
we may suppose that $b$, $b'$ and $b''$ are the following points of
$e^\perp\cap\mathcal{Q}$:
$$b=(0,1,0,0,0),\quad b'=(0,0,0,1,0),\quad b''=(-1,1,0,n,n).$$
Therefore, we see that
$$\pi_1=\langle (1,0,0,0,0), e\rangle, \quad \pi_2=\langle (-1,0,0,n,0), e\rangle, \quad
\pi_3=\langle (-1,1,0,0,0), e\rangle.$$ Now a secant line of $\pi_1$ must have nonzero
projections to two of the first, third and fifth coordinates.  Let $v$ be a point of
$\mathcal{Q}$ not in $\pi_1$. Below we will identify lines with $2\times 5$ matrices,
whereby the rows of the matrix give us the row-space of the associated 2-dimensional
vector subspace.  Let $M$ be the Gram matrix for the bilinear form arising from $Q$.  Now
$v^\perp$ meets $\pi_1$ in a line $s_v$:
$$s_v=v^\perp\cap \pi_1=\left[\begin{smallmatrix}
-v_1&0&0&0&v_5\\
-2v_3&0&v_5&0&0
\end{smallmatrix}\right]$$
which is a secant if and only if $\det(s_vMs_v^T)$ is a square. We have
$$\det(s_vMs_v^T)=\det
\left[\begin{smallmatrix}
-2v_1v_5&-2v_3v_5\\
-2v_3v_5&2v_5^2
\end{smallmatrix}\right]=-4v_1v_5^3-4v_3^2v_5^2=-4v_5^2(v_1v_5+v_3^2)=4v_5^2(v_2v_4).$$
and hence $v_2v_4\in\square$. 
Now
$$v^\perp\cap \pi_2=\left[\begin{smallmatrix}
2v_3&0&-(nv_1+v_5)&0&2nv_3\\
-2v_3&0&-(nv_2-v_5)&2nv_3&0
\end{smallmatrix}\right]\quad\text{ and }\quad v^\perp\cap \pi_3=\left[\begin{smallmatrix}
nv_1+v_4&nv_1+v_5&0&0&n(v_4-v_5)\\
2v_3&-2v_3&v_4-v_5&0&0
\end{smallmatrix}\right].$$
We want to show that $\det((v^\perp\cap \pi_2)M(v^\perp\cap \pi_2)^T)$ and
$\det((v^\perp\cap \pi_3)M(v^\perp\cap \pi_3)^T)$ are not both square and not both
non-square.  These determinants turn out to be respectively:
\begin{align*}
\Delta_2&=16 n^2 v_3^2 (n v_1 v_2 + n v_2^2 - v_3^2 - v_1 v_5 - v_2 v_5)=16 n^2 v_3^2 v_2 (n (v_1  +  v_2) +v_4 -  v_5)\\
\Delta_3&=4 n (v_4 - v_5)^2 (-n (v_1v_5+v_3^2) + n v_1 v_4 + v_4^2  - v_4 v_5)
=4 n (v_4 - v_5)^2 v_4(n (v_1+  v_2)  + v_4  - v_5).
\end{align*}
So the product of $\Delta_2$ and $\Delta_3$ is
$$8^2 n^3 v_3^2 (v_4 - v_5)^2 (n (v_1 + v_2) +v_4 - v_5)^2(v_2v_4),$$ and since
$v_2v_4\in\square$, we see that $\Delta_2\Delta_3$ is simply a product of $n$ and a
square.  Therefore $\Delta_2\Delta_3$ is a non-square and hence the lines $v^\perp\cap
\pi_2$ and $v^\perp\cap \pi_3$ are of different isometry types with respect to the quadric
$\mathcal{Q}$; thus completing the proof.
\end{proof}

\begin{definition}\label{defn:eqrelation}
Let $\mathcal{B}$ be a BLT-set of lines of $\mathsf{W}(3,q)$ and let $\ell$ be a line of
$\mathsf{W}(3,q)$ disjoint from every member of $\mathcal{B}$.  Define a binary relation
$\equiv_\ell$ on $\mathcal{B}$ by setting $b\equiv_\ell b'$ if and only if
$$b=b'\text{  or   }\mathrm{Tr}(b,\ell)\cap\mathrm{Tr}(b',\ell)=\varnothing.$$
\end{definition}

\begin{lemma}\label{lemma:eqrelation}
The relation $\equiv_\ell$ given by Definition \ref{defn:eqrelation} is an equivalence
relation with two equivalence classes of equal size.
\end{lemma}

\begin{proof}
First note that $\equiv_\ell$ is trivially reflexive and symmetric, so we will prove that
$\equiv_\ell$ is transitive (by using the contrapositive statement of the standard
definition).  Suppose we have three elements $b,b',b''$ of $\mathcal{B}$ such that
$\mathrm{Tr}(b,\ell)\cap\mathrm{Tr}(b'',\ell)\ne \varnothing$, that is, there is some line
$m$ meeting $\ell$, $b$ and $b''$. To establish transitivity, we will show that
$\mathrm{Tr}(b',\ell)$ must meet either $\mathrm{Tr}(b,\ell)$ or
$\mathrm{Tr}(b'',\ell)$. Let us look at the dual situation in the parabolic quadric
$\mathsf{Q}(4,q)$, and write $x^D$ for the dual object corresponding to $x$.

\begin{figure}[H]
\centering
\includegraphics[scale=0.6]{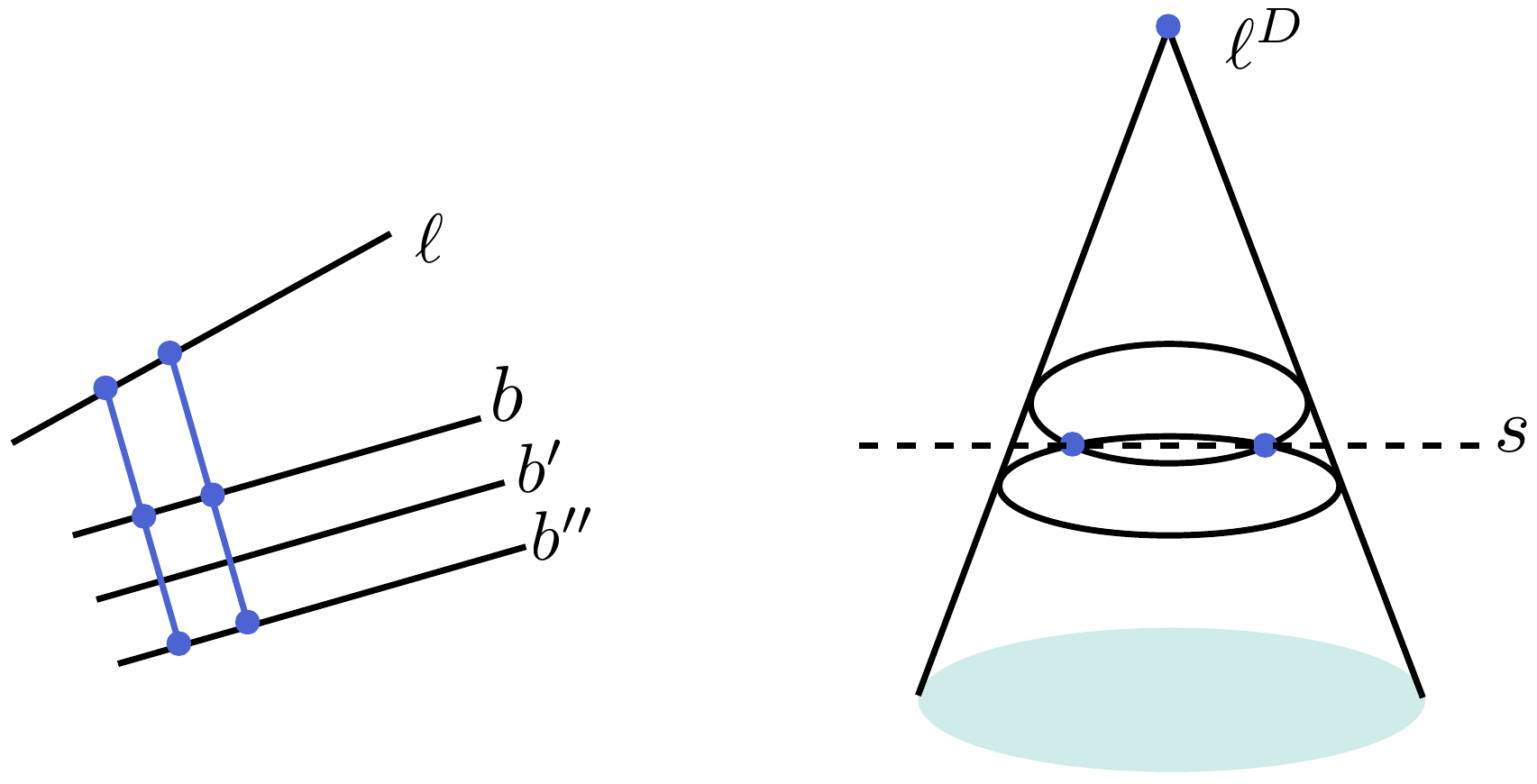}
\caption{The dual situation.}
\end{figure}

Then $\ell$ corresponds to a point $\ell^D$, $\mathrm{Tr}(b,\ell)$ and
$\mathrm{Tr}(b'',\ell)$ translate to conic sections of the cone $C$ on the perp of
$\ell^D$, and the two elements of $\mathrm{Tr}(b,\ell)\cap\mathrm{Tr}(b'',\ell)$ give us
the two points of the secant line $s$ to the conic given by the intersection of the perps
of the points $b^D$ and $(b'')^D$ dual to $b$ and $b''$. Let $\pi_2$ and $\pi_3$ be the
conics arising from taking the perps of the lines $\langle b^D, (b')^D\rangle$ and
$\langle (b')^D, (b'')^D\rangle$.  Then by Lemma \ref{inthequadric}, $C$ meets $\pi_2$ and
$\pi_3$ in non-degenerate lines of different types.  This means that the perp of $\ell^D$
must meet one of $\pi_2$ and $\pi_3$ in a secant line of $\mathsf{Q}(4,q)$, which is
equivalent to $\mathrm{Tr}(b',\ell)$ meeting either $\mathrm{Tr}(b,\ell)$ or
$\mathrm{Tr}(b'',\ell)$ in two elements.  Therefore $\equiv_\ell$ is transitive, and
hence, an equivalence relation.

Now for an arbitrary pair of distinct elements $b,b'\in\mathcal{B}$, we know from Lemma
\ref{grid0or2} that the cardinality of $\mathrm{Tr}(b,\ell)\cap\mathrm{Tr}(b',\ell)$ is
$0$ or $2$. So the number of elements of $\mathcal{B}$ which are not equivalent to $b$
under $\equiv_\ell$ is $(q+1)/2$. Therefore, the size of the equivalence class of $b$ is
also $(q+1)/2$ and we have completed the proof.
\end{proof}

%
%

\section{A construction of hemisystems of flock generalised quadrangles}\label{section:construction}

\begin{construction}\label{construction}
Consider a flock generalised quadrangle $\Knarr(\mathcal{O})$ in the Knarr model with base
point $P$ and null polarity $\perp$.  Fix a totally isotropic line $\ell$ in $P^\perp$
that is disjoint from every element of $\mathcal{O}$.  Suppose there is a partition of
$\mathcal{O}$ into two halves $\mathcal{O}^+$, $\mathcal{O}^-$.  For a subset
$\mathcal{S}$ of the totally isotropic planes on $\ell$, let
$\mathcal{L}^{\pm}_{\mathcal{S}}$ be the lines\footnote{Recall that the lines of
  $\Knarr(\mathcal{O})$ are each totally isotropic planes of $\mathsf{W}(5,q)$.} of
$\Knarr(\mathcal{O})$ that meet some element of $\mathcal{O}^{\pm}$ in a line, and which
meet some element of $\mathcal{S}$ in a point.  For any point $X\notin P^\perp$ and
$\pi\in\mathcal{O}$, define $X^\pi$ by
$$X^\pi:= \langle \ell,\ell^\perp\cap \langle X,X^\perp\cap\pi\rangle\rangle.$$ Note that
any $X^\pi$ is a totally isotropic plane on $\ell$. 
\end{construction}

\begin{theorem}\label{theorem:construction}
Let $\mathcal{S}$ be a subset of size $(q-1)/2$ of the totally isotropic planes on $\ell$
except $\langle P,\ell\rangle$, and let $\mathcal{S}^c$ be the complementary set of
$(q+1)/2$ planes on $\ell$ (except $\langle P,\ell\rangle$).  In Construction
\ref{construction}, the set
$$\mathcal{O}^+\cup \mathcal{L}^+_{\mathcal{S}}\cup \mathcal{L}^-_{\mathcal{S}^c}$$ of
lines of $\Knarr(\mathcal{O})$ is a hemisystem if and only if the following regularity
condition holds for any point $X\notin P^\perp, \ell^\perp$:
$$|\{ \pi\in \mathcal{O}^+ ~\mid~ X^\pi\in \mathcal{S}\}| = |\{ \pi\in \mathcal{O}^- ~\mid~
X^\pi\in \mathcal{S}\}|.$$
\end{theorem}

\begin{proof}
For each point $X$ not in $P^\perp$, we can associate a vector $\tilde{X}$ as follows:
$$\tilde{X}:= [ a, a^c ], \quad a=|\{\pi\in \mathcal{O}~\mid~X^\pi\in \mathcal{S}\}|,\quad
a^c=|\{\pi\in \mathcal{O}~\mid~X^\pi\in \mathcal{S}^c\}|.$$ This allows us to partition
the points of $\Knarr(\mathcal{O})$ so that we can easily construct a tactical
decomposition matrix as follows.  Consider the matrix whose columns are indexed by
$\mathcal{O}^+$, $\mathcal{L}^+_{\mathcal{S}}$, $\mathcal{L}^+_{\mathcal{S}^c}$,
$\mathcal{L}^-_{\mathcal{S}}$, $\mathcal{L}^-_{\mathcal{S}^c}$, and whose rows are indexed
by a partition of the points of $\Knarr(\mathcal{O})$ which we describe as follows.  The
single element $P$ is one part of the partition, and the points of type (ii) split into
two parts depending if the given point is in some element of $\mathcal{O}^+$ or
$\mathcal{O}^-$. (Note that such a point cannot belong to two elements of
$\mathcal{O}$). For the points of type (i), those in $\ell^\perp$ are partitioned by
$\mathcal{S}$ and $\mathcal{S}^c$.  For those outside of $\ell^\perp$, we partition the
points $X$ according to their vector value $\tilde{X}$.

To show that $\mathcal{O}^+\cup \mathcal{L}^+_{\mathcal{S}}\cup
\mathcal{L}^-_{\mathcal{S}^c}$ is a hemisystem, we need to show that the corresponding
columns in this matrix add to the constant column-vector with each entry equal to
$(q+1)/2$.  The initial $1+2+2$ rows of the table can be filled in easily, and we then
have
\begin{center}\footnotesize
\begin{tabular}{p{3cm}|cc|cccc}
&$\mathcal{O}^+$&$\mathcal{O}^-$& $\mathcal{L}^+_{\mathcal{S}}$& 
$\mathcal{L}^+_{\mathcal{S}^c}$& $\mathcal{L}^-_{\mathcal{S}}$& $\mathcal{L}^-_{\mathcal{S}^c}$\\
\hline
$P$&$\half(q+1)$&$\half(q+1)$&0&0&0&0\\
Point (i) in elt. $\mathcal{S}$&0&0&$\half(q+1)$&0&$\half(q+1)$&0\\
Point (i) in elt. $\mathcal{S}^c$&0&0&0&$\half(q+1)$&0&$\half(q+1)$\\
Point (ii) in elt. $\mathcal{O}^+$&1&0&$\half(q-1)$&$\half(q+1)$&0&0\\
Point (ii) in elt. $\mathcal{O}^-$&0&1&0&0&$\half(q-1)$&$\half(q+1)$ \\
\hline
Point (i) $X$ not in $\ell^\perp$,
and $\tilde{X}=[a,a^c]$&0&0&?&?&?&?\\
\end{tabular}
\end{center}
So now we are left with the submatrix defined by the points of type (i) that are not in
$\ell^\perp$.  Let us consider one row of this submatrix, representing those points $X$
which have a constant value for $\tilde{X}$, say $[a,a^c]$.  First notice that $a+a^c=q+1$
and hence by the regularity condition of the hypothesis, we know that the corresponding
row of the matrix is

\begin{center}
\begin{tabular}{l|cc|cccc}
&$\mathcal{O}^+$&$\mathcal{O}^-$& $\mathcal{L}^+_{\mathcal{S}}$& $\mathcal{L}^+_{\mathcal{S}^c}$&
 $\mathcal{L}^-_{\mathcal{S}}$& $\mathcal{L}^-_{\mathcal{S}^c}$\\
\hline
$X\notin \ell^\perp$,
$\tilde{X}=[a,a^c]$&0&0&$\half a$&$\half a^c$&$\half a$&$\half a^c$\\
\end{tabular}
\end{center}
It therefore follows that $\mathcal{O}^+\cup \mathcal{L}^+_{\mathcal{S}}\cup
\mathcal{L}^-_{\mathcal{S}^c}$ is a hemisystem.
\end{proof}

\begin{corollary}\label{constructioncorollary}
Suppose in Construction \ref{construction}, that for all $X\notin P^\perp,\ell^\perp$,
$$\{X^\pi\mid\pi\in\mathcal{O}^+\}=\{X^\pi\mid\pi\in\mathcal{O}^-\}.$$ Then the set
$\mathcal{O}^+\cup \mathcal{L}^+_{\mathcal{S}}\cup \mathcal{L}^-_{\mathcal{S}^c}$ of lines
of $\Knarr(\mathcal{O})$ is a hemisystem for any choice of $\mathcal{S}$ with
$|{\mathcal{S}}|=(q-1)/2$.
\end{corollary}

The objective in the remainder of this section is to find a method for finding a partition
of $\mathcal{O}$ and a line $\ell$ such that the condition of Corollary
\ref{constructioncorollary} is satisfied.

\begin{remark}
It is not known whether every hemisystem arising from Theorem \ref{theorem:construction}
must also satisfy the stronger condition of the above corollary.
\end{remark}

Now we look to the generalised quadrangle $\mathsf{W}(3,q)$ and certain totally isotropic
lines which respect a partition of the BLT-set $\mathcal{B}$. Recall from Section
\ref{flocks} that any line of $\mathsf{W}(3,q)$ not in $\mathcal{B}$ must be concurrent
with $0$ or $2$ elements of $\mathcal{B}$. So in particular, if we fix a point $X$ not in
any element of $\mathcal{B}$, then the map
$$b\mapsto \langle X,b\cap X^\perp\rangle$$
from $\mathcal{B}$ to lines on $X$ is two-to-one (see also \cite[Lemma 3.4]{BaderDuranteLawLunardonPenttila}).

\begin{definition}
Let $\mathcal{B}^+$ and $\mathcal{B}^-$ be a partition of a BLT-set $\mathcal{B}$ of lines
of $\mathsf{W}(3,q)$. Then a totally isotropic line $\ell$ not meeting any element of
$\mathcal{B}$ is \textbf{compatible} with the partition $\left\{\mathcal{B}^+,
\mathcal{B}^-\right\}$ if for all $X\in\ell$, we have $\{ \langle X, b \cap X^\perp
\rangle\mid b\in\mathcal{B}^+\}=\{ \langle X, b \cap X^\perp \rangle\mid
b\in\mathcal{B}^-\}$.
\end{definition}

We can identify when Corollary \ref{constructioncorollary} holds by just looking to the
quotient symplectic space $P^\perp/P$.

\begin{lemma}\label{constructioncorollary2}
Consider a flock generalised quadrangle $\Knarr(\mathcal{O})$ in the Knarr model with base
point $P$ and null polarity $\perp$.  Fix a totally isotropic line $\ell$ in $P^\perp$
that is disjoint from every element of $\mathcal{O}$.  Suppose there is a partition of
$\mathcal{O}$ into two halves $\mathcal{O}^+$ and $\mathcal{O}^-$.  Then the following two
conditions are equivalent:
\begin{enumerate}
\item[(i)] For all $E\notin P^\perp,\ell^\perp$, we have
  $\{E^\pi\mid\pi\in\mathcal{O}^+\}=\{E^\pi\mid\pi\in\mathcal{O}^-\}$.
\item[(ii)] The line $\langle P,\ell\rangle/ P$ of the quotient symplectic space $P^\perp
  / P$ is compatible with the following partition of the projection of $\mathcal{O}$ to
  $P^\perp/P$:
$$\mathcal{O}^+/P:=\{\pi/P\mid\pi\in\mathcal{O}^+\},\quad\mathcal{O}^-/P:=\{\pi/P\mid\pi\in\mathcal{O}^-\}.$$
\end{enumerate}
\end{lemma}

\begin{proof}
Let $E\notin P^\perp,\ell^\perp$.  We will show that the two sets of objects are in
bijective correspondence, and to do this, we break the proof down into steps. These steps
require us to define six sets of objects as follows. Let:
\begin{itemize}
\item $\textsf{Planes}(P)$ be the set of planes on $P$ not meeting $\ell$;
\item $\textsf{Lines}(P^\perp\cap\langle\ell^\perp, E\rangle)$ be the set of lines on $P$
  in the $3$-space $P^\perp\cap\langle\ell^\perp, E\rangle$, meeting $\langle
  P,\ell\rangle$ in $P$;
\item $\textsf{Planes}(P^\perp\cap\langle\ell^\perp, E\rangle)$ be the set of planes on
  $P$ inside $P^\perp\cap\langle \ell^\perp, E\rangle$ meeting $\ell$ in a point;
\item $\textsf{Planes}(E)$ be the set of planes on $E$ not meeting $\langle
  P,\ell\rangle$;
\item $\textsf{Points}(E^\perp\cap\ell^\perp)$ be the set of points of the plane
  $E^\perp\cap\ell^\perp$ not in $P^\perp$;
\item $\textsf{Planes}(\ell)$ be the set of planes on $\ell$, not on $P$.
\end{itemize}

In fact, we will eventually establish that the diagram below commutes, and it will serve
as a guide to showing that $\Phi$ on the bottom line is bijective by proving that
particular maps are surjections whose preimages yield regular partitions:
$$
\begin{CD}
\textsf{Planes}(P) @> \pi\mapsto \langle E,E^\perp\cap \pi\rangle >> \textsf{Planes}(E) \\
@VV\pi\mapsto \pi\cap\langle\ell^\perp,E\rangle V @VV\sigma\mapsto \sigma\cap\ell^\perp V\\
\textsf{Lines}(P^\perp\cap\langle\ell^\perp, E\rangle)@>m\mapsto \ell^\perp\cap\langle 
E,E^\perp\cap m\rangle >>  \textsf{Points}(E^\perp\cap\ell^\perp)\\
@VVm\mapsto \langle m,m^\perp\cap \ell^\perp\rangle V @VV X\mapsto \langle\ell, X\rangle V\\
\textsf{Planes}(P^\perp\cap\langle\ell^\perp, E\rangle)@>
\Phi >>  \textsf{Planes}(\ell)\\
\end{CD}$$
where $\Phi$ is defined by
$$\Phi:\langle m, m^\perp\cap \ell^\perp\rangle\rightarrow \langle
\ell,\ell^\perp\cap\langle E,E^\perp\cap \pi_m\rangle\rangle, \quad (\pi_m:=\langle m,
m^\perp\cap \ell^\perp\rangle).$$

Let us look at these maps in turn. It is not difficult to see that the map $\pi\mapsto
\pi\cap\langle\ell^\perp,E\rangle$ is a $q$-to-one surjective map from
$\textsf{Planes}(P)$ onto $\textsf{Lines}(P^\perp\cap\langle\ell^\perp, E\rangle)$, as
each element of $\textsf{Lines}(P^\perp\cap\langle\ell^\perp, E\rangle)$ is incident with
$q$ elements of $\textsf{Planes}(P)$. The map $m\mapsto \langle m,m^\perp\cap
\ell^\perp\rangle$ is also a $q$-to-one surjective map, but this time from
$\textsf{Lines}(P^\perp\cap\langle\ell^\perp, E\rangle)$ onto
$\textsf{Planes}(P^\perp\cap\langle\ell^\perp, E\rangle)$, since each element of
$\textsf{Planes}(P^\perp\cap\langle\ell^\perp, E\rangle)$ contains $q$ elements of
$\textsf{Lines}(P^\perp\cap\langle\ell^\perp, E\rangle)$.  So the composition of these
maps, which we see in the diagram as ``down from the top left to the bottom left'', is a
$q^2$-to-one surjective map from $\textsf{Planes}(P)$ onto
$\textsf{Planes}(P^\perp\cap\langle\ell^\perp, E\rangle)$. If we can show that the map
$\pi\mapsto \langle\ell, \ell^\perp\cap \langle E,E^\perp\cap \pi\rangle\rangle$ is a
$q^2$-to-one surjection, then $\Phi$ will be well-defined and bijective. Now the map
$\pi\mapsto \langle E,E^\perp\cap \pi\rangle$ is injective since if $\langle E,E^\perp\cap
\pi_1\rangle=\langle E,E^\perp\cap \pi_2\rangle$ for two distinct planes $\pi_1$ and
$\pi_2$ , then $E^\perp\cap\pi_1$ and $E^\perp\cap\pi_2$ are coplanar and hence meet in at
least a point, which implies that $P\in E^\perp$; a contradiction.  Now we work our way
down the right hand side of the diagram. The map $\sigma\mapsto \sigma\cap\ell^\perp$ is
clearly a $q$-to-one surjection (as every point of
$\textsf{Points}(E^\perp\cap\ell^\perp)$ is incident with $q$ elements of
$\textsf{Planes}(E)$), and the map $X\mapsto \langle \ell, X\rangle$ from
$\textsf{Points}(E^\perp\cap\ell^\perp)$ to $\textsf{Planes}(\ell)$ is also a $q$-to-one
surjection. So we see that the composition of the maps traversing downwards on the
right-hand side of the diagram is a $q^2$-to-one surjection.  Therefore it follows that
the left-to-right maps in the diagram are each bijections, and in particular, $\Phi$ is a
bijection.

So we have established that condition (i) of the hypothesis is equivalent to the condition
that for all $E\notin P^\perp,\ell^\perp$ we have
\begin{equation}\label{condition}
\{\langle m,m^\perp\cap\ell^\perp \rangle ~\mid~ m=\pi\cap
\langle\ell^\perp,E\rangle,\pi\in\mathcal{O}^+\}=\{\langle m,m^\perp\cap\ell^\perp\rangle
~\mid~ m=\pi\cap \langle\ell^\perp,E\rangle, \pi\in\mathcal{O}^-\}.
\end{equation}
Now we project by $P$ to the symplectic space $P^\perp/P$ (which is isomorphic to
$\mathsf{W}(3,q)$). The image of $\mathcal{O}$ under the projection is a BLT-set of lines
of $P^\perp /P$. The plane $\langle P,\ell\rangle$ is mapped to a line which does not meet
any element of the BLT-set of lines. If $\pi\in\mathcal{O}$, then the line
$m=\pi\cap\langle\ell^\perp, E\rangle$ projects to the unique point on $\pi/P$ which is
collinear to the point $(\ell\cap E^\perp)/P$ on $\langle P,\ell\rangle/P$.
\begin{figure}[H]
\centering
\includegraphics[scale=0.6]{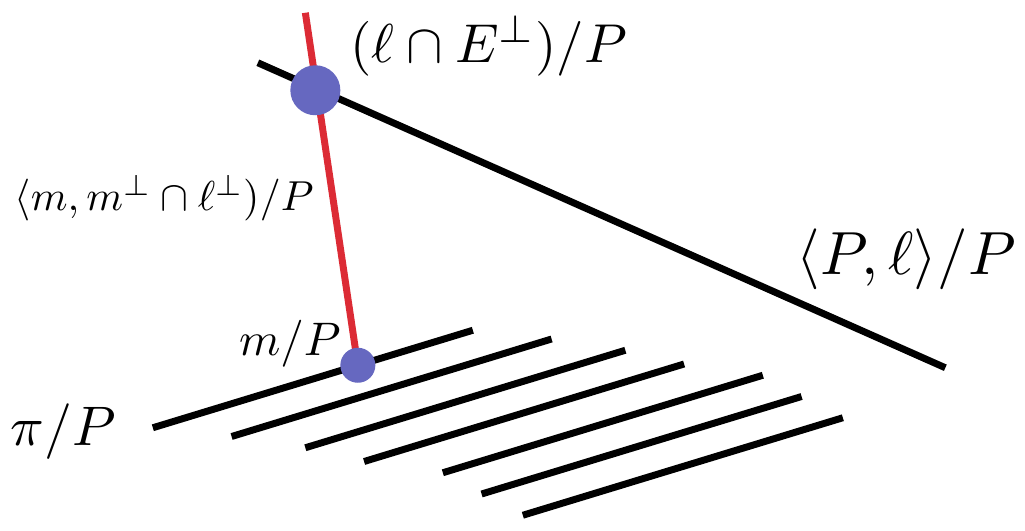}
\caption{The projection to $P^\perp/P$.}
\end{figure}

The choice of $E$ corresponds to the choice of the point on $\langle P,\ell\rangle /P$. So
the condition (\ref{condition}) is equivalent to having, for all points $X$ on the line
$\langle P,\ell\rangle /P$,
$$\{ \langle X, b \cap X^\perp \rangle\mid b\in\mathcal{O}^+/P\}=\{ \langle X, b \cap
X^\perp \rangle\mid b\in\mathcal{O}^-/P\},$$ that is, $\langle P,\ell\rangle /P$ is
compatible with the partition $\{\mathcal{O}^+/P,\mathcal{O}^-/P\}$.
\end{proof}

The remaining ingredient for the proof of Theorem \ref{thm:intro} is the following:

\begin{theorem}\label{maintheorem}
Consider a set $\mathcal{O}$ of totally isotropic planes of $\mathsf{W}(5,q)$ each
incident with a point $P$ such that
$$\mathcal{B}:=\{\pi/P\mid \pi\in\mathcal{O}\}$$ is a BLT-set of lines of the quotient
symplectic space $P^\perp/P\cong\mathsf{W}(3,q)$.  Suppose, furthermore, that we have a
line $\ell$ of $P^\perp/P$ not meeting any element of $\mathcal{B}$, and let $\equiv_\ell$
be the binary relation on $\mathcal{B}$ given in Definition \ref{defn:eqrelation}.  Then
$\ell$ is compatible with the equivalence classes of $\equiv_\ell$.
\end{theorem}

\begin{proof}
Let $\mathcal{B}^+$ and $\mathcal{B}^-$ be the two equivalence classes of $\equiv_\ell$.
Let $E$ be a point on $\ell$, and let $m$ be a line on $E$ concurrent with two distinct
elements $b$ and $b'$ of $\mathcal{B}$.  Since $m$ is both an element of
$\mathrm{Tr}(b,\ell)$ and $\mathrm{Tr}(b',\ell)$, it is clear that $b\not\equiv_\ell b'$;
that is $b$ and $b'$ lie in different $\equiv_\ell$-classes. Hence whenever we have a line
$m$ concurrent with $\ell$ and concurrent with two elements $b$ and $b'$ of $\mathcal{B}$,
then $m=\langle E, b \cap E^\perp \rangle=\langle E, b' \cap E^\perp \rangle$ where
$E=m\cap \ell$, and we can assume that $b\in\mathcal{B}^+$ and $b'\in\mathcal{B}^-$.  So
it follows, that for all $E\in\ell$, we have $\{ \langle E, b \cap E^\perp \rangle\mid
b\in\mathcal{B}^+\}=\{ \langle E, b' \cap E^\perp \rangle\mid b'\in\mathcal{B}^-\}$.
Therefore $\ell$ is compatible with $\{\mathcal{B}^+,\mathcal{B}^-\}$.
\end{proof}

So together with Corollary \ref{constructioncorollary} and Lemma
\ref{constructioncorollary2}, we obtain a hemisystem of the associated flock generalised
quadrangle and hence we have proved Theorem \ref{thm:intro}.

\begin{remark}\
\begin{enumerate}
\item
The hemisystem one obtains from Corollary \ref{constructioncorollary} depends on the
choice of a set $\mathcal{S}$ of $(q-1)/2$ totally isotropic planes on a line $\ell$ (in
$P^\perp$ disjoint from every element of $\mathcal{O}$).  For a set $\mathcal{O}$ and line
$\ell$ as above, we can obtain different hemisystems by the choice of $\mathcal{S}$, which
we explore further in \cite{BambergGiudiciRoyleprep}.

\item
To our knowledge, the only possible partition $\{\mathcal{O}^+,\mathcal{O}^-\}$ that works
for a given line $\ell$ is the one arising from the equivalence relation $\equiv_\ell$. We
have found no counter-examples in small flock generalised quadrangles.
\end{enumerate}
\end{remark}

%
%

\section{The Cossidente-Penttila hemisystems}\label{section:cphemi}

As we mentioned in the introduction, Cossidente and Penttila \cite{CossidentePenttila05}
produced for each odd prime power $q$, a hemisystem of $\mathsf{H}(3,q^2)$ admitting
$\mathsf{P\Omega}^-(4,q)$. Their construction can be summarised as follows. Consider an
elliptic quadric $\mathcal{E}\cong \mathsf{Q}^-(3,q)$ whose defining polarity commutes
with the polarity defining a fixed Hermitian variety $\mathsf{H}(3,q^2)$. We then see that
all the $q^2+1$ singular points of $\mathcal{E}$ are totally isotropic points of
$\mathsf{H}(3,q^2)$, and each totally isotropic line of $\mathsf{H}(3,q^2)$ meets
$\mathcal{E}$ in either no points or one point (see the first line of the proof of
\cite[Proposition 2.3]{CossidentePenttila05}). We may call such lines \textit{external}
and \textit{tangent} respectively. The key ingredient is the action of the perfect group
$\mathsf{P\Omega}^-(4,q)$ on the lines of $\mathsf{H}(3,q^2)$; it has two orbits on
tangents and two orbits on externals (this follows from the proof of \cite[Theorem
  3.1]{CossidentePenttila05}).  Take one of each and their union is a hemisystem admitting
$\mathsf{P\Omega}^-(4,q)$. In fact, it is clear that this is the only way one can obtain a
hemisystem admitting this group $\mathsf{P\Omega}^-(4,q)$ stabilising $\mathcal{E}$. We
will show in this section that the Cossidente-Penttila hemisystems can be constructed from
Theorem \ref{construction}, and we first interpret the original construction in the
following way:

\begin{lemma}\label{lemma:cphemi}
Let $q$ be an odd prime power and consider an elliptic quadric
$\mathcal{E}:=\mathsf{Q}^-(3,q)$ and Hermitian variety $\mathsf{H}(3,q^2)$ whose defining
polarities commute. Let $G$ be the subgroup $\mathsf{P\Omega}^-(4,q)$ stabilising
$\mathcal{E}$ and let $\Omega$ be one of the two orbits of size $q^2(q^2-1)/2$ of $G$ on
totally isotropic lines of $\mathsf{H}(3,q^2)$ external to $\mathcal{E}$.  Fix a point
$P\in\mathcal{E}$ and let $\mathcal{P}^+$ and $\mathcal{P}^-$ be the two orbits of $G_P$
on lines on $P$.  Let $\mathcal{M}^+$ be the set of lines of $\Omega$ which meet a unique
element of $\mathcal{P}^+$ in a point. Let $\mathcal{M}^-$ be comprised of the set of
lines of $\Omega$ which meet a unique element of $\mathcal{P}^-$ in a point, together with
the tangent lines which meet a unique element of $\mathcal{P}^-$ in a point.  Then
$$\mathcal{P}^+\cup \mathcal{M}^+\cup\mathcal{M}^-$$ is a hemisystem of
$\mathsf{H}(3,q^2)$ admitting $G$, and hence, is equivalent to a Cossidente-Penttila
hemisystem.
\end{lemma}

\begin{proof}
Let $\mathcal{M}^-_T$ be the tangent lines of $\mathcal{M}^-$ and let $\mathcal{M}^-_E$ be
the external lines of $\mathcal{M}^-$.  We will show that $\mathcal{P}^+\cup
\mathcal{M}^-_T$ is a $G$-orbit on tangents and that $\mathcal{M}^-_E\cup \mathcal{M}^+$
is a $G$-orbit on externals.  Clearly $\Omega=\mathcal{M}^-_E\cup\mathcal{M}^+$ as every
line of $\Omega$ must meet $P^\perp$ and hence is concurrent with some element of
$\mathcal{P}^+$ or $\mathcal{P}^-$. Also, $\mathcal{M}^-_T$ is an orbit of $G_P$ on
tangents. Let $Y$ be a point of $\mathsf{H}(3,q^2)$ not in $\mathcal{E}$ but ``nearby'' to
$\mathcal{E}$, that is, there is some point $X$ of $\mathcal{E}$ such that $Y$ is
collinear with $X$ in $\mathsf{H}(3,q^2)$.  Now there are $q^2(q^2+1)(q+1)/2$ ``nearby''
points in total (by the last paragraph of the proof of \cite[Proposition
  2.2]{CossidentePenttila05}), there are $q^2+1$ points of $\mathcal{E}$ and any point of
$\mathcal{E}$ is collinear to $q^2(q+1)$ other points of $\mathsf{H}(3,q^2)$.  So there
are exactly two lines on $Y$ which are tangent to $\mathcal{E}$, and these two lines are
in different $G$-orbits. Now suppose we have an element $m\in\mathcal{M}_T^-$. Then $m$ is
concurrent with a unique line on $P$ in a point $Y$, say, and the two tangents on $Y$
(namely $m$ and $\langle P,Y\rangle$) are in different $G$-orbits and $\langle
P,Y\rangle\in\mathcal{P}^-$ by the definition of $\mathcal{M}_T^-$. Therefore,
$\mathcal{P}^+\cup \mathcal{M}^-_T$ is contained in a $G$-orbit, but since it has
cardinality half the number of tangents, $\mathcal{P}^+\cup \mathcal{M}^-_T$ is a
$G$-orbit on tangents.  So we see that $\mathcal{P}^+\cup \mathcal{M}^+\cup\mathcal{M}^-$
is the union of a $G$-orbit on externals and a $G$-orbit on tangents, and so it follows
that it is equivalent to a Cossidente-Penttila hemisystem.
\end{proof}

\subsection{The Knarr model is a generalisation of field reduction}\label{KnarrFieldReduction}

One can view the Knarr model of a flock generalised quadrangle as a generalisation of the
field reduction of $\mathsf{H}(3,q^2)$ to $\mathsf{W}(7,q)$, and we will take some time
now to explain how this works. Let $\Phi$ be the map from $\mathsf{H}(3,q^2)$ to
$\mathsf{W}(7,q)$ obtained by identifying $\GF(q^2)$ with $\GF(q)^2$ such that the
four-dimensional vector space over $\GF(q^2)$ becomes an eight-dimensional vector space
over $\GF(q)$.  Let $\xi$ be an element of $\GF(q^2)$ not in $\GF(q)$ whose relative trace
$\mathsf{T}_{q^2\rightarrow q}(\xi)=\xi+\xi^q$ is zero.  Then the following form is
bilinear and alternating over $\GF(q)$:
$$B(u,v) = \mathsf{T}_{q^2\rightarrow q}(\xi \cdot \beta( u,v ))$$ where $\beta$ is the
original Hermitian form over $\GF(q^2)$.  Under this correspondence points go to lines,
lines go to solids, and so forth. Fix a point $P$ of $\mathsf{H}(3,q^2)$ and the $q+1$
Hermitian lines $\mathcal{P}$ incident with $P$. Then under $\Phi$, we obtain a symplectic
line $P^\Phi$ of $\mathsf{W}(7,q)$ and $q+1$ symplectic solids $\mathcal{P}^\Phi$. Now let
$\Pi$ be a hyperplane of $\PG(7,q)$ such that $\Pi^\perp$ is a point on $P^\Phi$. Consider
the following projection map from totally isotropic subspaces of $\mathsf{W}(7,q)$ to
totally isotropic subspaces of the quotient polar space
$\Pi/\Pi^\perp\cong\mathsf{W}(5,q)$:
\begin{equation}\label{rhoequation}
\rho:U\mapsto (\Pi\cap \langle \Pi^\perp, U\rangle)/\Pi^\perp.
\end{equation}
The composition of $\Phi$ and $\rho$ then gives us the Knarr model of $\mathsf{H}(3,q^2)$:
\begin{center}
\begin{table}[H]\footnotesize
\begin{tabular}{l|l}
Objects in $\mathsf{H}(3,q^2)$&Objects in $\mathsf{W}(5,q)$\\ \hline \\ $P$&Point
$P^{\Phi\rho}$ of $\mathsf{W}(5,q)$\\ Lines on $P$&BLT-set of totally isotropic planes on
$P^{\Phi\rho}$\\ Points collinear with $P$& Lines incident with some BLT-set element, not
incident with $P^{\Phi\rho}$\\ Points not collinear with $P$& Points not collinear with
$P^{\Phi\rho}$\\ Lines not on $P$&Planes, not on $P^{\Phi\rho}$, which meet some BLT-set
element in a line\\ \\ \hline
\end{tabular}
\caption{The Knarr model of $\mathsf{H}(3,q^2)$ obtained via field reduction and projection.}
\end{table}
\end{center}

\subsection{The elliptic quadric in $\mathsf{W}(7,q)$}\label{tensor}

It will be useful in understanding and proving Theorem \ref{cphemi2} to geometrically
describe the image of the elliptic quadric $\mathcal{E}$ in $\mathsf{W}(7,q)$.  Let $V$ be
a 4-dimensional vector space over $\GF(q^2)$ with $q$ odd, equipped with a non-degenerate
Hermitian form $\beta$. Let $\{v_1,v_2,v_3,v_4\}$ be a basis for $V$ such that the
determinant of the Gram matrix of $\beta$ with respect to this basis is a nonsquare in
$\GF(q)$. (For example, take an orthonormal basis and multiply the first basis element by
a primitive element of $\GF(q^2)$.)  Let $U$ be the $\GF(q)$-span of
$\{v_1,v_2,v_3,v_4\}$. Then $V=U\otimes \GF(q^2)$ and the restriction $\beta_1$ of $\beta$
to $U$ is a symmetric bilinear form associated with an elliptic quadratic form.  Then
$g\in O^-(4,q)$ acts on $V$ such that $(u\otimes \lambda)^g=u^g\otimes \lambda$ for all
$g\in O^-(4,q)$, $u\in U$ and $\lambda\in\GF(q^2)$.  This action preserves the form
$\beta$ and so we have an embedding of isometry groups $O^-(4,q)\leqslant U(4,q^2)$.

Now recall from Section \ref{KnarrFieldReduction} that $V$ is also an 8-dimensional vector
space over $\GF(q)$ and is equipped with a non-degenerate alternating form $B=
\mathsf{T}_{q^2\rightarrow q}(\xi \cdot\beta)$ where $ \mathsf{T}_{q^2\rightarrow
  q}(\xi)=0$. Now for $u\otimes\lambda,w\otimes\mu\in U\otimes\GF(q^2)$, we have
\begin{align*}
B(u\otimes\lambda,w\otimes\mu)&= \mathsf{T}_{q^2\rightarrow q}(\xi\cdot \beta(u\otimes\lambda,w\otimes\mu))\\
                              &= \mathsf{T}_{q^2\rightarrow q}(\xi\lambda\mu^q\cdot \beta(u\otimes 1,w\otimes1))\\
                             &=\beta_1(u,w) \mathsf{T}_{q^2\rightarrow q}(\xi\lambda\mu^q)
\end{align*}
since $\beta_1(u,w)\in\GF(q)$. Moreover, $\beta_2(\lambda,\mu)= \mathsf{T}_{q^2\rightarrow
  q}(\xi\lambda\mu^q)$ is an alternating form on the 2-dimensional space $W=\GF(q^2)$ as
$\lambda\mu^q$ defines a Hermitian form on $W$. From now on we consider $W$ as a
2-dimensional vector space over $\GF(q)$ such that $V=U\otimes W$. Also, given the
symmetric form $\beta_1$ on $U$ and the alternating form $\beta_2$ on $W$ we have
$$B(u_1\otimes w_1,u_2\otimes w_2)=\beta_1(u_1,u_2)\beta_2(w_1,w_2).$$ Then the central
product $O^-(4,q)\circ\mathsf{Sp}(2,q)$ acts naturally on $V$ via $(u\otimes
w)^{(g,h)}=u^g\otimes w^h$ for all $g\in O^-(4,q),h\in \mathsf{Sp}(2,q),u\in U$ and $w\in
W$, and this action preserves the form $B$.

\subsection{A Knarr model of the Cossidente-Penttila hemisystems}

By the following result, we see how one can obtain the Cossidente-Penttila hemisystems via
Theorem \ref{construction}.  First we summarise in a table the notation which we are using
in this section:

\begin{center}\footnotesize
\begin{tabular}{lp{12cm}}
\hline
Object&Description\\
\hline
$\mathcal{E}$&an elliptic quadric embedded in $\mathsf{H}(3,q^2)$\\
$G$&subgroup $\mathsf{P\Omega}^-(4,q)$ stabilising $\mathcal{E}$\\
$\Omega$& a $G$-orbit on
totally isotropic lines of $\mathsf{H}(3,q^2)$ external to $\mathcal{E}$\\
$\mathcal{P}$&lines on $P$\\
$\mathcal{P}^+$, $\mathcal{P}^-$& $G$-orbits on $\mathcal{P}$\\
$\mathcal{M}^+$&set of lines of $\Omega$ which meet an element of $\mathcal{P}^+$ in a point\\
$\mathcal{M}^-$&set of
lines of $\Omega$ which meet an element of
$\mathcal{P}^-$ in a point, together with the tangent lines which an element of
$\mathcal{P}^-$ in a point\\
$\mathcal{O}$&$\mathcal{P}^{\Phi\rho}$\\
$\mathcal{O}^+$&$(\mathcal{P}^+)^{\Phi\rho}$\\
$\mathcal{O}^-$&$(\mathcal{P}^-)^{\Phi\rho}$\\
\hline
\end{tabular}
\end{center}

\begin{theorem}\label{cphemi2}
Consider the $q^2+1$ points $\mathcal{E}^{\Phi\rho}$ of $\mathsf{W}(5,q)$ obtained by
field reduction to $\mathsf{W}(7,q)$ and projection to $\mathsf{W}(5,q)$ as described
above. Then $\mathcal{E}^{\Phi\rho}$ spans a solid whose image under the null polarity is
a totally isotropic line $\ell$ of $\mathsf{W}(5,q)$.  For a subset $\mathcal{S}$ of the
totally isotropic planes on $\ell$, let $\mathcal{L}^{\pm}_{\mathcal{S}}$ be the totally
isotropic planes of $\mathsf{W}(5,q)$ that meet some element of $\mathcal{O}^{\pm}$ in a
line, and that meet some element of $\mathcal{S}$ in a point.  Let $\mathcal{R}$ be the
set of (t.i.) planes on $\ell$ that meet some element of $(\mathcal{M}^+)^{\Phi\rho}$ in a
point.  Then $\mathcal{R}$ has size $(q-1)/2$,
$(\mathcal{M}^+)^{\Phi\rho}=\mathcal{L}_{\mathcal{R}}^+$ and
$(\mathcal{M}^-)^{\Phi\rho}=\mathcal{L}_{\mathcal{R}^c}^-$.
\end{theorem}

\begin{proof}
First we show that $\mathcal{E}^{\Phi\rho}$ spans a solid whose perp is a totally
isotropic line of $\mathsf{W}(5,q)$.  Recall from the discussion in Section \ref{tensor}
that we can write our $8$-dimensional vector space $V$ as $U\otimes W$ where $U$ carries a
symmetric bilinear form of minus type (and is $4$-dimensional) and $W$ carries an
alternating bilinear form (and is $2$-dimensional) such that the ``product'' of these two
forms defines a non-degenerate alternating bilinear form on $V$.  Let $u\in U$ be a
totally singular vector with respect to the elliptic form on $U$. Then $\langle
u\rangle\otimes W$ is a totally isotropic line in $V$ with respect to $B$, and so we will
identify $P^{\Phi\rho}$ with $(\langle u\rangle\otimes W)^\rho$.  Fix $w\in W$ and let
$X=\langle u\otimes w\rangle$. Then $X^{\perp}$ is a hyperplane containing
$u^{\perp}\otimes W$. We can also find an element $\overline{u}\in U$ such that
$\beta(\overline{u},u)\neq 0$.  Then $B(\overline{u}\otimes w,u\otimes
w)=\beta_1(\overline{u},u)\beta_2(w,w)=0$. Hence $X^{\perp}= \langle \overline{u}\otimes
w,u^{\perp}\otimes W\rangle$.

Now we may write the image of our elliptic quadric as 
$$\mathcal{E}^\Phi=\{\langle u'\rangle\otimes W\mid u'\in U,  \beta_1(u',u')=0\}.$$
Then for $\langle u'\rangle\otimes W\in \mathcal{E}$ we have
$$ X^{\perp}\cap (\langle u'\rangle\otimes W)=
\begin{cases}
                        \langle u'\rangle\otimes W & \text{ if }u'\in u^{\perp}\\
                        \langle u'\otimes w\rangle & \text{ if }u'\notin u^\perp 
\end{cases}.$$
The only totally singular point in $u^\perp$ is $u$, so 
$$\langle \{X^\perp \cap E\mid E\in\mathcal{E}^\Phi\}\rangle= \langle u\otimes W,
u'\otimes w\mid u'\notin u^\perp\rangle=\langle U\otimes w,u\otimes W\rangle,$$ which is a
4-dimensional projective space whose perp is a totally isotropic plane of
$\mathsf{W}(7,q)$.  So returning to our model, we can substitute in this argument
$X^\perp$ for the the hyperplane $\Pi$ in the Knarr model (\ref{rhoequation}) and we see
that $\mathcal{E}^{\Phi\rho}$ spans a solid whose perp is a totally isotropic line $\ell$
of $\mathsf{W}(5,q)$.

Now consider the set $\mathcal{R}$ of totally isotropic planes on $\ell$ that meet some
element of $(\mathcal{M}^+)^{\Phi\rho}$ in a point. First note that for every line $m$ of
$\mathsf{H}(3,q^2)$, we have that $m^{\Phi\rho}$ meets $\ell^\perp$ in a point (n.b.,
$\perp$ now denotes the polarity defining $\mathsf{W}(5,q)$).  Moreover, if $m$ is a
tangent line to $\mathcal{E}$, then
$m^{\Phi\rho}\cap\ell^\perp\in\mathcal{E}^{\Phi\rho}$. Now $\mathcal{E}^\Phi$ contains the
line $P^\Phi=\langle u\rangle\otimes W$, and the other lines $\langle u'\rangle\otimes W$
meet the totally isotropic solid $\langle U\otimes w\rangle$ in a point. So
$\mathcal{E}^{\Phi\rho}$ consists precisely of the point $P^{\Phi\rho}$ and the points of
the totally isotropic plane $\langle U\otimes w\rangle^\rho$ which are not on $\ell$. This
plane $\langle U\otimes w\rangle^\rho$ is incident with $\ell$, so we know that just two
of the planes on $\ell$ meet the images of tangent lines under $\Phi\rho$, whereas the
remaining $q-1$ planes meet the images of external lines. It is then not difficult to
ascertain that half of these $q-1$ planes are the elements of $\mathcal{R}$.

By definition, $\mathcal{O}^+$ is the image of $\mathcal{P}^+$ under $\Phi\rho$.  Let
$m\in\mathcal{M}^+$, that is, $m\in\Omega$ and $m$ meets some element of
$b\in\mathcal{P}^+$ in a point $E$.  Then $m^\Phi$ is a solid meeting the solid $b^\Phi$
in the line $E^\Phi$. So $m^{\Phi\rho}$ is a plane meeting an element of $\mathcal{O}^+$
in $E^{\Phi\rho}$, the latter being a line as $E^\Phi\le X^\perp$. Now $m^{\Phi\rho}$
meets $\ell^\perp$ in a point, and so there is a unique totally isotropic plane on $\ell$
which meets $m^{\Phi\rho}$ in a point, namely $\langle \ell, \ell^\perp\cap
m^{\Phi\rho}\rangle$.  By definition, this plane must be an element of $\mathcal{R}$, and
hence, $m^{\Phi\rho}\in \mathcal{L}_{\mathcal{R}}^+$.  It is not difficult to show that
the cardinalities of $\mathcal{M}^+$ and $\mathcal{L}_{\mathcal{R}}^+$ are equal to
$q^2(q^2-1)/4$ and hence $(\mathcal{M}^+)^{\Phi\rho}=\mathcal{L}_{\mathcal{R}}^+$.  By a
similar argument, we have that $(\mathcal{M}^-)^{\Phi\rho}=\mathcal{L}_{\mathcal{R}^c}^-$.
\end{proof}

\begin{corollary}
The Cossidente-Penttila hemisystem of $\mathsf{H}(3,q^2)$ can be described as
$\mathcal{O}^+\cup \mathcal{L}_{\mathcal{R}}^+\cup \mathcal{L}_{\mathcal{R}^c}^-$ in the
Knarr model $\mathcal{K}(\mathcal{O})$ of $\mathsf{H}(3,q^2)$, that is, the
Cossidente-Penttila hemisystems occur as examples of output of our Construction
\ref{construction}.
\end{corollary}

\newpage

%
%

\section{Appendix: A proof that a hemisystem yields a partial quadrangle}

Let $\mathcal{Q}$ be a generalised quadrangle of order $(s,s^2)$, $s>1$ odd.  If
$\mathcal{P}$ is the set of points of $\mathcal{Q}$, then the algebra
$\mathbb{C}\mathcal{P}$ of functions from $\mathcal{P}$ to $\mathbb{C}$ decomposes into an
orthogonal decomposition of three irreducible subspaces $V^0$, $V^+$ and $V^-$, which are
in fact eigenspaces for the adjacency matrix of $\mathcal{Q}$ with eigenvalues $s(s^2+1)$,
$s-1$ and $-s^2-1$ respectively.  Moreover, $V^0$ is just the subspace of constant
functions from $\mathcal{P}$ to $\mathbb{C}$, which is spanned by the all-one function
$j$.  We will be considering multisets of elements of $\mathcal{P}$, that is, functions
from $\mathcal{P}$ to the non-negative integers $\mathbb{N}_0$.  To make our notation
simpler, we will avoid using the symbol ``chi'' and a subscript for the characteristic
function of a multiset, rather we will simply use the multiset itself if it is clear that
we are using it as an element of $\mathbb{C}\mathcal{P}$.

A \textit{weighted tight set} is any multiset $v$ of $\mathbb{N}_0\mathcal{P}$ such that
$v\in V^0\bigoplus V^+$. Similarly, a \textit{weighted $m$-ovoid} is any multiset residing
in $V^0\bigoplus V^-$.  An example of a weighted $m$-ovoid is a \textit{hemisystem of
  points} of a generalised quadrangle; the dual notion of a hemisystem of lines (see
\cite[Lemma 1]{BKLP07}).  The scalar product $A\scalarproduct B$ of two multisets $A$ and
$B$ gives us the function $i\mapsto A(i)B(i)$, which one can think of as the
generalisation of the intersection of two sets. Therefore, we write $|A|$ for
$A\scalarproduct j$. So suppose we have a weighted tight set $v^+$ and a weighted
$m$-ovoid $v^-$. Then by orthogonal projection, we see that
$$
v^+-\frac{|v^+|}{|\mathcal{P}|}j\in V^+\quad\text{ and }\quad
v^--\frac{|v^-|}{|\mathcal{P}|}j\in V^-
$$ and hence these two vectors are orthogonal to one another. This observation then
implies that $v^+$ and $v^-$ meet in a number of elements determined by their sizes:
$$v^+\scalarproduct v^- = \frac{|v^+||v^-|}{|\mathcal{P}|}.$$ (Note that $|\mathcal{P}|$ is simply
the number $(s+1)(s^3+1)$ of points of $\mathcal{Q}$).

\begin{lemma}
Let $x$ and $y$ be two noncollinear points of $\mathcal{Q}$ and consider the multiset
$\mathcal{T}$ of points consisting of $\{x,y\}^\perp$ and $s$ copies of $x$ and $y$. Then
$\mathcal{T}$ is a weighted tight set.
\end{lemma}

\begin{proof}
Consider the function $v:=(s^3+1)\mathcal{T}-(s+1)j$.  We will show that $v$ is an
eigenvector of the adjacency matrix $A$ with eigenvalue $s-1$.  Now we collect a few
formulae which are simple to demonstrate with geometric arguments, plus the fact that the
size of $\{x,y,z\}^{\perp}$ is $s+1$, where $x,y,z$ are pairwise noncollinear (see
\cite{BoseShrikhande72}):
\begin{align*}
jA&=s(s^2+1)j,\\
{\{x,y\}^\perp}A&=(s+1)j-s\cdot( {x^\perp}+{y^\perp})+(s-1)\cdot {\{x,y\}^\perp}+s^2\cdot {\{x,y\}},\\
{\{x,y\}}A&={x^\perp}+{y^\perp}-{\{x,y\}}.
\end{align*}
So
\begin{align*}
\mathcal{T}A&={\{x,y\}^\perp}A+s\cdot (xA)+s\cdot(yA)\\
&=(s+1)j+(s-1){\{x,y\}^\perp}+(s^2-s){\{x,y\}}\\
&=(s+1)j+(s-1)\mathcal{T}
\end{align*}
It then follows that $v$ is an eigenvector of the adjacency matrix $A$ with eigenvalue $s-1$.
\end{proof}

\begin{theorem}
Let $\mathcal{H}$ be a set of points of $\mathcal{Q}$ giving rise to a hemisystem of the
dual of $\mathcal{Q}$. Then the geometry $\Gamma$ obtained by restricting the point-set of
$\mathcal{Q}$ to $\mathcal{H}$ is a partial quadrangle.
\end{theorem}

\begin{proof}
Clearly every two points of $\Gamma$ share at most one line and there are no triangles in
$\Gamma$, so it suffices to prove the ``$\mu$ condition''. Let $x,y$ be two arbitrary
noncollinear points in $\mathcal{H}$ and let $\mu$ be the number of points of
$\mathcal{H}$ which are collinear to both $x$ and $y$.  Note that $x^\perp\cap y^\perp$
consists of $s^2+1$ points in $\mathcal{Q}$.  We simply use the above lemma to show that
$\mu$ is independent of the choice of $x$ and $y$.  Consider the multiset $\mathcal{T}$ of
points consisting of $\{x,y\}^\perp$ and $s$ copies of $x$ and $y$.  So in other words,
this multiset can be expressed by $\mathcal{T}={\{x,y\}^\perp}+s(\{x\}+\{y\})$.  Simply
note now that
$$
\frac{|\mathcal{T}||\mathcal{H}|}{(s+1)(s^3+1)}=\mathcal{T}\scalarproduct \mathcal{H}
={\{x,y\}^\perp}\scalarproduct\mathcal{H}+s(\{x\}\scalarproduct\mathcal{H})+s(\{y\}\scalarproduct\mathcal{H})\\
=\mu+2s
$$
and hence $\mu=\tfrac{1}{2}((s^2+1)+2s)-2s=(s-1)^2/2$.
Therefore $\Gamma$ is a partial quadrangle.
\end{proof}

\address{
   \noindent J. Bamberg, M. Giudici and G. F. Royle\\
   School of Mathematics and Statistics,\\
   The University of Western Australia,\\
   35 Stirling Highway, Crawley, W.A. 6009\\
  Australia\\
   \email{bamberg@maths.uwa.edu.au\\
   giudici@maths.uwa.edu.au\\
   gordon@maths.uwa.edu.au}}


\end{document}